\documentclass[10pt,a4paper]{article}

\usepackage[T1]{fontenc}
\usepackage[utf8]{inputenc}
\usepackage{amsmath,amssymb,amsthm,amsfonts}
\usepackage{a4wide}

\theoremstyle{plain}
\newtheorem{theorem}{Theorem}[section]
\newtheorem{proposition}[theorem]{Proposition}
\newtheorem{lemma}[theorem]{Lemma}
\newtheorem{corollary}[theorem]{Corollary}
\theoremstyle{definition}

\newtheorem{remark}[theorem]{Remark}

\theoremstyle{remark}
\numberwithin{equation}{section}

\usepackage[pdfborder={0 0 0},bookmarksopen]{hyperref}
\hypersetup{%
	pdftitle = {Minimal Conditions for Implications of Gronwall--Bellman Type},
	pdfauthor = {Martin Herdegen, Sebastian Herrmann},
	pdfpagemode = UseNone,
}

\newcommand{\NN}{\mathbb{N}}
\newcommand{\RR}{\mathbb{R}}
\newcommand{\ZZ}{\mathbb{Z}}

\newcommand{\cA}{\mathcal{A}}
\newcommand{\cB}{\mathcal{B}}

\newcommand{\diff}{\mathrm{d}}
\newcommand{\dd}{\,\mathrm{d}}

\newcommand{\1}{\mathbf{1}}

\newcommand*{\ol}[1]{\overline{#1}}

\newcommand*{\sef}{\mathrm{sf}}

\begin{document}
\title{%
Minimal Conditions for Implications of\\Gronwall--Bellman Type\thanks{The authors are very grateful for the referee's pertinent remarks on the original manuscript.}%
}
\date{\today}
\author{%
  Martin Herdegen%
  \thanks{
  Department of Statistics, University of Warwick, Conventry, CV4 7AL, UK, email
  \href{mailto:M.Herdegen@warwick.ac.uk}{\nolinkurl{M.Herdegen@warwick.ac.uk}}.
  Financial support from the Swiss National Science Foundation (SNF) through grant SNF 105218\_150101 is
gratefully acknowledged.
  }
  \and
  Sebastian Herrmann%
  \thanks{
  Department of Mathematics, University of Michigan, 530 Church Street, Ann Arbor, MI 48109, USA, email
  \href{mailto:sherrma@umich.edu}{\nolinkurl{sherrma@umich.edu}}.
  Financial support by the Swiss Finance Institute is gratefully acknowledged.
  }
}
\maketitle

\begin{abstract}
Gronwall--Bellman type inequalities entail the following implication: if a sufficiently integrable function satisfies a certain homogeneous linear integral inequality, then it is nonpositive. We present a minimal (necessary and sufficient) condition on the Borel measure underlying the integrals for this implication to hold. The condition is also a necessary prerequisite for any nontrivial bound on solutions to inhomogeneous linear integral inequalities of Gronwall--Bellman type.
\end{abstract}

\vspace{0.5em} 

{\small

\vspace{0.25em}
}

%=================================================================================
\section{Introduction}
%=================================================================================
\label{sec:introduction}

It is the purpose of this paper to characterise all Borel measures $\mu$ on $\RR$ for which the following implication holds:
\begin{align}
\label{eqn:implication}
y(t)
\leq \int_{(-\infty,t)} y\dd\mu
\text{ for }\mu\text{-a.e.~}t\in\RR
\;\;\Longrightarrow\;\;
y(t)
\leq 0\text{ for }\mu\text{-a.e.~}t\in\RR.
\end{align}
Here, $y$ is any Borel function on $\RR$ with $\int_{(-\infty,t)} \vert y \vert \dd\mu < \infty$ for each $t \in \RR$, so that the integrals in \eqref{eqn:implication} exist.

Does the implication \eqref{eqn:implication} hold for all Borel measures? If $\mu$ is the restriction of the Lebesgue measure to an interval of the form $(c,d)$ for $-\infty < c < d \leq \infty$, then the implication \eqref{eqn:implication} holds and is a special case of the famous Gronwall--Bellman lemma (see, e.g., \cite[Lemma D.2]{SellYou2002}). However, if $\mu(\diff t)=\frac{1}{t-c}\1_{(c,d)}(t) \dd t$, then the function $y(t) = t-c$ satisfies the integral inequality in \eqref{eqn:implication}, but $y > 0$ $\mu$-a.e., so that \eqref{eqn:implication} fails in this case.

It turns out that the decisive difference between these two examples is that in the former, $\mu((c,t)) < \infty$ for some $t > c$, while in the latter, $\mu((c,t)) = \infty$ for every $t > c$. Indeed, our main result Theorem~\ref{thm:global} entails that the implication \eqref{eqn:implication} holds if and only if the following condition on the measure $\mu$ holds:
\begin{itemize}
\item[(M)] For each $a \in [-\infty,\infty)$, there is $t > a$ such that $\mu_\sef((a,t)) < \infty$.
\end{itemize}
Here, $\mu_\sef$ is the so-called semi-finite part of $\mu$. (If $\mu$ is semi-finite as in the two examples above, then $\mu_\sef = \mu$; cf.~Proposition~\ref{prop:semi-finite}.)

In addition to this ``global'' result, we also provide a ``local'' version, which states that for fixed $a \in[-\infty,\infty)$, the ``localised'' implication
\begin{align}
\label{eqn:implication:local}
\exists b > a:
\left[
y(t)
\leq \int_{(a,t)} y\dd\mu
\text{ for }\mu\text{-a.e.~}t\in(a,b)
\;\;\Longrightarrow\;\;
y(t)
\leq 0\text{ for }\mu\text{-a.e.~}t\in(a,b)
\right]
\end{align}
is equivalent to the existence of $t > a$ such that the semi-finite part of $\mu$ puts finite mass on $(a,t)$. A consequence of our main result is that condition (M) is also necessary for any nontrivial bound on solutions to \emph{inhomogeneous} linear integral inequalities; see the last paragraph of Section~\ref{sec:main results}.

Bounds on solutions to integral or differential inequalities are an important tool for the analysis of various integral or differential equations.\footnote{For instance, using standard arguments (see, e.g., \cite[Section I.3.1]{BainovSimeonov1992} for the classic case or \cite[Section 4]{Rao1979} for the case of Lebesgue--Stieltjes integrals), our main result can be used to derive a uniqueness result for integral equations with Borel measures satisfying condition (M).} The classic results of Gronwall~\cite{Gronwall1919}, Reid~\cite{Reid1930}, and Bellman~\cite{Bellman1943} have been extended in many different ways over the past century; we refer to \cite{ChandraFleishman1970} and to the monograph \cite{BainovSimeonov1992} for an overview and an extensive list of references. While most of the extant results stay within the realm of ordinary Riemann integration, also other integrals are considered in the literature: Riemann--Stieltjes integrals \cite{Jones1964, ErbeKong1990}, modified Stieltjes integrals \cite{SchmaedekeSell1968}, abstract Stieltjes integrals \cite{Herod1969, Mingarelli1981}, Lebesgue--Stieltjes integrals \cite{Pandit1978, Rao1979, DasSharma1980}, and integrals on general measure spaces. In particular, very general results on Gronwall--Bellmann type inequalities for general measure spaces were obtained by Horv\'
ath~\cite{Horvath1996, Horvath1999, Horvath2003} (see also Gy\H{o}ri and Horv\'ath~\cite{GyoriHorvath1997}). However, in the special case of the homogeneous linear integral inequality considered in \eqref{eqn:implication}, our condition (M) is still weaker than the conditions imposed in \cite[Theorem 3.1]{Horvath1996}.

The remainder of the article is organised as follows. Section~\ref{sec:main results} states and discusses our main results. Section~\ref{sec:auxiliary results} contains auxiliary results. The proofs of our main results are in Section~\ref{sec:proofs}.

%=================================================================================
\section{Main results and ramifications}
%=================================================================================
\label{sec:main results}

Before we can state our main results, we need to introduce the so-called semi-finite part of a measure.

\paragraph{Semi-finite measures.}
Fix a measure space $(X,\Sigma,\mu)$. Recall that $\mu$ is called \emph{semi-finite} if for every $E\in\Sigma$ with $\mu(E) = \infty$, there is $F\in\Sigma$ such that $F \subset E$ and $0 < \mu(F) < \infty$ \cite[Definition 211F]{Fremlin2003}. As in \cite[Exercise 213X~(c)]{Fremlin2003}, the \emph{semi-finite part} of $\mu$ is the measure $\mu_\sef: \Sigma \to [0,\infty]$ given by
\begin{align*}
\mu_\sef(E)
&= \sup \lbrace\mu(E\cap F) : F \in \Sigma,\,\mu(F) <\infty \rbrace.
\end{align*}
The following proposition collects basic facts about $\mu_\sef$. We omit the proofs (see \cite[Lemma 213A and Exercise 213X~(c)]{Fremlin2003}).

\begin{proposition}
\label{prop:semi-finite}
Let $(X,\Sigma,\mu)$ be a measure space.
\begin{enumerate}
\item $\mu_\sef$ is a semi-finite measure on $(X,\Sigma)$ and absolutely continuous with respect to $\mu$.
\item Any $\mu$-integrable real-valued function $f$ is $\mu_\sef$-integrable, with the same integral. 
\item[\rm (b$^\prime$)] For any $\mu_\sef$-integrable real-valued function $f_\sef$, there exists a $\mu$-integrable real-valued function $f$ such that $f = f_\sef$ $\mu_\sef$-a.e.
\item[\normalfont (c)] If $\mu$ is semi-finite, then $\mu = \mu_\sef$ and for every $E \in \Sigma$,
\begin{align*}
\mu(E)
&= \sup\lbrace \mu(F) : F\in\Sigma,\,F\subset E,\,\mu(F) < \infty \rbrace.
\end{align*}
\end{enumerate}
\end{proposition}

For any measurable real-valued function on $X$ and any Borel set $B\subset \RR$, $f \in B$ $\mu$-a.e.~implies $f \in B$ $\mu_\sef$-a.e.~by Proposition~\ref{prop:semi-finite}~(a). If $f$ is $\mu$-integrable and if $B$ contains $0$, then  also the converse implication holds, as the following corollary shows.

\begin{corollary}
\label{cor:semi-finite}
Let $(X,\Sigma,\mu)$ be a measure space, $f$ a $\mu$-integrable real-valued function, and $B\subset\RR$ a Borel set containing $0$. Then $f \in B$ $\mu$-a.e.~if and only if $f \in B$ $\mu_\sef$-a.e.
\end{corollary}

\begin{proof}
Suppose that $f \in B$ $\mu_\sef$-a.e. Then $\int_X \vert f \vert \1_{\lbrace f\in B^c\rbrace} \dd\mu_\sef = 0$, and as $\vert f \vert \1_{\lbrace f\in B^c\rbrace}$ is $\mu$-integrable, also $\int_X \vert f \vert \1_{\lbrace f\in B^c\rbrace} \dd\mu = 0$ by Proposition~\ref{prop:semi-finite}~(b). Thus, $\vert f \vert \1_{\lbrace f\in B^c\rbrace}=0$ $\mu$-a.e., and as $0 \not\in B^c$, we conclude that $f \in B$ $\mu$-a.e. The converse implication follows from Proposition~\ref{prop:semi-finite}~(a).
\end{proof}

\paragraph{Main results.}
From now on, we work on the real line $\RR$ with the standard topology and its Borel $\sigma$-algebra $\cB(\RR)$. Our main result comes in a ``local'' and a ``global'' version. For a fixed $a \in [-\infty,\infty)$, the following ``local'' result characterises all Borel measures on $\RR$ for which there exists $b\in(a,\infty)$ such that any solution to the integral inequality \eqref{eqn:thm:local:inequality} is nonpositive $\mu$-a.e.~on $(a,b)$.

\begin{theorem}[Local version]
\label{thm:local}
Let $\mu$ be a Borel measure on $\RR$. Then for all $a\in[-\infty,\infty)$, the following are equivalent:
\begin{itemize}
\item[\normalfont(M$_a$)] There is $t > a$ such that $\mu_\sef((a,t)) < \infty$.

\item[\normalfont(I$_a$)] There is $b \in(a,\infty)$ such that for any real-valued Borel function $y$ on $\RR$ which is $\mu$-integrable over $(a,b)$ and satisfies
\begin{align}
\label{eqn:thm:local:inequality}
y(t)
&\leq \int_{(a,t)} y \dd \mu \quad \text{for } \mu\text{-a.e.~} t \in (a,b),
\end{align}
we have $y \leq 0$ $\mu$-a.e.~on $(a,b)$.
\end{itemize}
\end{theorem}

\begin{remark}
An inspection of the proof of Theorem~\ref{thm:local} shows that if (M$_a$) holds for some finite $t > a$, then (I$_a$) holds with $b=t$. However, if (I$_a$) holds for some $b > a$, then (M$_a$) need not hold for $t=b$ (but for \emph{some} $t \in (a,b]$).
\end{remark}

The ``global'' version of our main result characterises all Borel measures on $\RR$ for which any solution to the integral inequality \eqref{eqn:thm:global:inequality} is nonpositive $\mu$-a.e.~on $\RR$.

\begin{theorem}[Global version]
\label{thm:global}
Let $\mu$ be a Borel measure on $\RR$.
The following are equivalent:
\begin{enumerate}
\item[\normalfont(M)] For each $a\in[-\infty,\infty)$, there is $t > a$ such that $\mu_\sef((a,t)) <\infty$.

\item[\normalfont(I)] For any real-valued Borel function $y$ on $\RR$ such that $\int_{(-\infty,t)} \vert y\vert\dd \mu < \infty$ for all $t \in \RR$ and
\begin{align}
\label{eqn:thm:global:inequality}
y(t)
&\leq \int_{(-\infty,t)} y \dd \mu \quad \text{for } \mu\text{-a.e.~} t \in \RR,
\end{align}
we have $y \leq 0$ $\mu$-a.e.
\end{enumerate}
In this case, $\mu_\sef$ is even $\sigma$-finite.
\end{theorem}

A couple of remarks are in order.

\begin{remark}
\label{rem:nonnegative}
Statement (I) in Theorem~\ref{thm:global} is equivalent to the following statement:
\begin{enumerate}
\item[\normalfont(I$^\prime$)] For any \emph{nonnegative} Borel function $y$ on $\RR$ satisfying $\int_{(-\infty,t)}y \dd\mu < \infty$ for all $t \in \RR$ and \eqref{eqn:thm:global:inequality}, we have $y = 0$ $\mu$-a.e.
\end{enumerate}
Indeed, the implication ``(I) $\Rightarrow$ (I$^\prime$)'' is trivial. For the converse, suppose that (I$^\prime$) holds and let $y$ be as in (I). It follows from \eqref{eqn:thm:global:inequality} that the positive part $y^+ = \max(y,0)$ fulfils $y^+(t) \leq \int_{(-\infty,t)} y^+\dd\mu$ for $\mu$-a.e.~$t \in \RR$. Thus, by (I$^\prime$), $y^+ = 0$ $\mu$-a.e., which in turn gives $y \leq 0$ $\mu$-a.e.
\end{remark}

\begin{remark}
\label{rem:Borel subsets}
Theorem~\ref{thm:global} is formulated for integral inequalities on the whole real line. Statements for Borel subsets $E \subset \RR$ can easily be obtained by applying Theorem~\ref{thm:global} to the measure $\mu_E$ defined by $\mu_E(A) = \mu(A\cap E)$, $A \in \cB(\RR)$. A typical example is $E = [a,b]$ for some $a<b$ in $\RR$. In this case, if $\mu_{[a,b]}$ satisfies (M), then by Theorem~\ref{thm:global}, we have $y \leq 0$ $\mu$-a.e.~on $[a,b]$ for any real-valued Borel function $y$ satisfying $\int_{[a,b]} \vert y \vert \dd\mu < \infty$ and $y(t) \leq \int_{[a,t)} y \dd\mu$ for $\mu$-a.e.~$t \in [a,b]$.
\end{remark}

\begin{remark}
\label{rem:counter-examples}
We provide a couple of counter-examples that reject some potential weakenings of the statements (M) or (I).
\begin{enumerate}
\item If we replace ``$a \in [-\infty,\infty)$'' in (M) by ``$a\in\RR$'', then the implication ``(M) $\Rightarrow$ (I)'' breaks down. For example, let $\mu$ be the (semi-finite) Lebesgue measure on $\RR$. Then $\mu((a,a+1)) < \infty$ for each $a \in \RR$, but $y(t) = t^{-2}\1_{(-\infty,-1]}(t)$ is positive on $(-\infty,-1]$ and solves \eqref{eqn:thm:global:inequality}: 
\begin{align*}
\int_{(-\infty,t)} y \dd\mu
&= \int_{-\infty}^t s^{-2}\1_{(-\infty,-1]}(s)  \dd s
= -t^{-1}\1_{(-\infty,-1]}(t) + \1_{(-1,\infty)}(t)\\
&\geq t^{-2}\1_{(-\infty,-1]}(t)
=y(t), \quad t \in \RR.
\end{align*}

\item If in (I), we only require that the negative part $y^- = \max(-y,0)$ of $y$ is $\mu$-integrable over $(-\infty,t)$ for all $t \in \RR$, then the implication ``(M) $\Rightarrow$ (I)'' breaks down. For example, suppose that $\mu$ is the Lebesgue measure restricted to $(0,1)$. Then (M) holds, but $y(t) = \frac{1}{t}\1_{(0,1)}(t)$ is positive on $(0,1)$ and solves \eqref{eqn:thm:global:inequality} (the integral on the right-hand side is $+\infty$ for any $t > 0$).

\item If the inequality in \eqref{eqn:thm:global:inequality} is replaced by an equality, then the implication ``(I) $\Rightarrow$ (M)'' breaks down. For example, suppose that $\mu$ is the counting measure for the positive rational numbers. Assume that $y$ is as in (I) but even solves the integral \emph{equation}\footnote{If $y$ solves \eqref{eqn:rem:counter-examples:integral equation} only for $\mu$-a.e.~$t \in \RR$, then we can pass to the function $\tilde y(t) =  \int_{(-\infty,t)} y \dd \mu$, which satisfies $\tilde y = y$ $\mu$-a.e.~and solves the integral equation everywhere. The same argument then gives $\tilde y = 0$ on $\RR$ and hence $y = 0$ $\mu$-a.e.}
\begin{align}
\label{eqn:rem:counter-examples:integral equation}
y(t)
&= \int_{(-\infty,t)} y \dd \mu, \quad t \in \RR.
\end{align}
We claim that then $y = 0$ on $\RR$. By \eqref{eqn:rem:counter-examples:integral equation} and since $\mu$ is supported on $[0,\infty)$, we have $y = 0$ on $(-\infty, 0]$. Seeking a contradiction, suppose that $y(t) \neq 0$ for some $t > 0$. By Lemma~\ref{lem:integral equation}, $y$ is monotone, so $\vert y(u) \vert \geq \vert y(t) \vert > 0$ for all $u \geq t$. But then
\begin{align*}
\int_{(t,t+1)} \vert y \vert \dd\mu
&\geq \vert y(t) \vert \mu((t,t+1))
= \infty,
\end{align*}
contradicting the integrability of $y$. Therefore, $y = 0$ is the only solution to \eqref{eqn:rem:counter-examples:integral equation}. However, the counting measure $\mu$ is semi-finite and $\mu((a,t)) = \infty$ for all $0 \leq a < t$, so that (M) fails.
\end{enumerate}
\end{remark}

\paragraph{Nonexistence of bounds for solutions to inhomogeneous linear integral inequalities.}
Gronwall--Bellman type inequalities (we refer to \cite{BainovSimeonov1992} for an overview) establish, under certain sufficient conditions, upper bounds for solutions $y$ to various \emph{inhomogeneous} linear integral inequalities such as
\begin{align}
\label{eqn:rem:inhomogeneous:integral inequality}
y(t)
&\leq f(t) + \int_{(-\infty,t)} y \dd\mu\quad\text{for }\mu\text{-a.e.~}t \in \RR.
\end{align}
In other words, these results provide a real-valued function $b$ (depending only on $f$ and $\mu$) such that any solution to \eqref{eqn:rem:inhomogeneous:integral inequality} satisfies $y \leq b$ $\mu$-a.e.

One consequence of our main result is that if condition (M) fails, then no such bound can exist (except in the trivial case that no solution to \eqref{eqn:rem:inhomogeneous:integral inequality} exists). Indeed, in the case that (M) fails, there is by Theorem~\ref{thm:global} a real-valued solution $\tilde y$ to \eqref{eqn:thm:global:inequality} which is positive with positive $\mu$-measure. Let $y$ be a solution to \eqref{eqn:rem:inhomogeneous:integral inequality}. Then by linearity, for each $n\in\NN$, the function $y_n = y + n \tilde y$ also solves the inhomogeneous linear integral inequality \eqref{eqn:rem:inhomogeneous:integral inequality}, but $\lim_{n\to\infty} y_n = \infty$ on the set where $\tilde y$ is positive.

%=================================================================================
\section{Auxiliary results}
%=================================================================================
\label{sec:auxiliary results}

The proofs of Theorems~\ref{thm:local} and \ref{thm:global} make use of several lemmas, which we state and prove in this section. We begin with a simple sufficient criterion for $\sigma$-finiteness of a Borel measure on $\RR$.  We provide a full proof because we were unable to find a reference.

\begin{lemma}
\label{lem:sigma finite}
Let $\mu$ be a Borel measure on $\RR$. Suppose that for every $a \in \RR$, there is $\varepsilon>0$ such that $\mu([a,a+\varepsilon)) <\infty$. Then $\mu$ is $\sigma$-finite.
\end{lemma}

\begin{proof}
Using that $\RR$ can be covered by countably many sets of the form $[K,\infty)$, it suffices to show that $\mu$ is $\sigma$-finite on $[K,\infty)$ for each $K \in \RR$. So fix $K\in\RR$ and consider
\begin{align*}
t^\star
&= \sup \lbrace t \in[K,\infty]: \mu \text{ is $\sigma$-finite on } [K,t) \rbrace \in [K,\infty].
\end{align*}
By assumption, $\mu$ is finite on $[K, K+\varepsilon)$ for some $\varepsilon > 0$. Hence, $t^\star > K$ and there is a sequence $(t_n)_{n\in\NN}\subset [K,t^\star]$ increasing to $t^\star$ such that for each $n\in\NN$, $\mu$ is $\sigma$-finite on $[K,t_n)$. We can thus choose, for each $n\in\NN$, a sequence $(E^n_i)_{i \in \NN}$ of Borel sets such that $\bigcup_{i \in \NN} E^n_i = [K,t_n)$ and $\mu(E^n_i) < \infty$ for all $i \in \NN$. Noting that the countable collection $(E^n_i)_{i,n\in\NN}$ covers $[K,t^\star)$, we conclude that $\mu$ is $\sigma$-finite on $[K,t^\star)$. Seeking a contradiction, suppose that $t^\star < \infty$. Then by assumption, there is $\varepsilon > 0$ such that $\mu$ is finite on $[t^\star, t^\star + \varepsilon)$, so that $\mu$ is $\sigma$-finite even on $[K,t^\star+\varepsilon)$. This contradicts the definition of $t^\star$. Thus, $t^\star = \infty$ and $\mu$ is $\sigma$-finite on $[K,\infty)$.
\end{proof}

For further reference, we state a special case of the Gronwall--Bellman lemma for Borel measures on $\RR$. It follows from \cite[Theorem~3.1]{Horvath1996}.

\begin{lemma}
\label{lem:gronwall}
Let $\mu$ be a Borel measure on $\RR$ and fix $-\infty\leq a<b\leq \infty$. Suppose that $\mu((a,t)) <\infty$ for each $t \in (a,b)$. If $y:(a,b)\to\RR$ is $\mu$-integrable over $(a,t)$ for each $t \in (a,b)$ and
\begin{align}
\label{eqn:lem:gronwall:integral inequality}
y(t)
&\leq \int_{(a,t)} y \dd \mu \quad\text{for $\mu$-a.e.~}t \in (a,b),
\end{align}
then $y \leq 0$ $\mu$-a.e.~on $(a,b)$. More precisely, for each $t\in(a,b)$, $y(t) \leq 0$ whenever the inequality \eqref{eqn:lem:gronwall:integral inequality} holds.
\end{lemma}

%\begin{proof}
%By \cite[Theorem~3.1~(a)]{Horvath1996} and the assumptions on $\mu$ (in particular, $1$ is $\mu$-integrable over $(a,t)$ for each $t \in (a,b)$), the homogeneous linear integral equality $s(t) = \int_{(a,t)} s \dd\mu$, $t \in (a,b)$, has the unique solution $s = 0$ . Then, by \cite[Theorem~3.1~(d)]{Horvath1996} and the assumptions on $y$, for each $t \in (a,b)$, $y(t) \leq 0$ whenever $y(t) \leq \int_{(a,t)} y \dd \mu$.
%\end{proof}

Next, we construct a positive solution to the integral equation
\begin{align*}
y(t)
&= \int_{(a,t)}y\dd\mu,\quad t \in (a,b),
\end{align*}
in the special case where $\mu$ has a ``singularity'' only at the left endpoint of some interval $(a,b)$, i.e., $\mu((t,b)) = \infty$ if and only if $t = a$. This is the crucial ingredient for the implication ``(I$_a$) $\Rightarrow$ (M$_a$)'' of Theorem~\ref{thm:local} (which we prove by contraposition).

\begin{lemma}
\label{lem:construction}
Let $\mu$ be a Borel measure on $\RR$ and $-\infty\leq a<b < \infty$ such that for $t \in [a,b)$, $\mu((t,b)) = \infty$ if and only if $t = a$. Then there is a positive Borel function $y:(a,b)\to(0,1]$ such that 
\begin{align*}
y(t)
&= \int_{(a,t)}y \dd\mu, \quad t \in (a,b),
\quad \text{and}\quad
\lim_{t \uparrow\uparrow b} y(t) = 1.
\end{align*}
In particular, $y$ is nondecreasing, $\mu$-integrable over $(a,b)$, and satisfies $\lim_{t\downarrow\downarrow a} y(t) = 0$.
\end{lemma}

\begin{proof}
We start by constructing a candidate function. It follows from \cite[Theorem~3.1~(a)]{Horvath1996} that there is a function $y:(a,b)\to\RR$ such that $\int_{(t,b)}\vert y \vert \dd\mu < \infty$ for each $t \in (a,b)$ and
\begin{align}
\label{eqn:lem:construction:pf:y:1}
y(t)
&= \frac{1}{1+\mu(\lbrace t \rbrace)} - \frac{1}{1+\mu(\lbrace t \rbrace)} \int_{(t,b)} y \dd \mu, \quad t\in(a,b).
\end{align}
More precisely, in the notation of \cite{Horvath1996}, we take $X = (a,b)$, $\cA$ the Borel $\sigma$-algebra on $(a,b)$, $\mu$ the restriction of ``our'' $\mu$ to $(a,b)$, and $S(x) = (x,b)$, $x \in (a,b)$. Then the function $S: X \to \cA$ satisfies condition (C) of \cite[Section 2]{Horvath1996}. The functions $f$ and $g$ of \cite[Theorem~3.1]{Horvath1996} are in our case given by $f(x) = -g(x) = 1/(1+\mu(\lbrace x \rbrace))$, $x \in (a,b)$. So $\vert f (x)\vert = \vert g(x)\vert \leq 1$, $x \in (a,b)$, and $1$ is $\mu$-integrable over $S(x) = (x,b)$ for each $x \in (a,b)$ by assumption. Then $f$ and $g$ are $\mu$-integrable provided that they are Borel measurable.\footnote{We thank the referee for pointing out and providing a proof for this measurability requirement.} This is true if $x \mapsto \mu(\lbrace x \rbrace)$, $x \in (a,b)$, is Borel measurable. So let $(t_n)_{n \in \NN} \subset (a,b)$ be a sequence decreasing to $a$. Then for any $\varepsilon > 0$, the sets $A^\varepsilon_n = \lbrace x \in (t_n, b) : \mu(\lbrace x \rbrace) > \varepsilon \rbrace$, $n \in \NN$, are finite since $\mu((t,b)) < \infty$ for $t \in (a,b)$ by assumption. It follows that $A^\varepsilon = \lbrace x \in (a,b) : \mu(\lbrace x \rbrace) > \varepsilon \rbrace = \bigcup_{n \in \NN} A^\varepsilon_n$ is a countable set, and hence it is Borel measurable. This implies that $\lbrace x \in (a,b) : \mu(\lbrace x \rbrace) > 0 \rbrace = \bigcup_{n \in \NN} A^{1/n}$ is also Borel measurable. We conclude that $x \mapsto \mu(\lbrace x \rbrace)$, $x \in (a,b)$, is Borel measurable.

Multiplying \eqref{eqn:lem:construction:pf:y:1} by $(1+\mu(\lbrace t \rbrace))$ and then subtracting $y(t)\mu(\lbrace t \rbrace)$ on both sides yields
\begin{align}
\label{eqn:lem:construction:pf:y:2}
y(t)
&= 1 - \int_{(t,b)} y \dd \mu - y(t) \mu(\lbrace t \rbrace)
= 1- \int_{[t,b)} y \dd \mu, \quad t\in(a,b).
\end{align}

We now show that $y$ has the asserted properties. First, letting $t$ increase to $b$ in \eqref{eqn:lem:construction:pf:y:2} and using dominated convergence gives $y(b-)=\lim_{t \uparrow\uparrow b} y(t)= 1$.

Second, we show that $y$ is positive. Seeking a contradiction, suppose that $y(t) \leq 0$ for some $t \in (a,b)$. Then $t^\star
:= \sup \lbrace t \in (a,b) : y(t) \leq 0 \rbrace \in (a,b)$ (note that $y(b-) = 1$ implies $t^\star < b$).
Let $(t_n)_{n\in\NN} \subset (a,t^\star]$ be a sequence increasing to $t^\star$ such that $y(t_n) \leq 0$ for all $n\in\NN$. We claim that $y(t^\star) \leq 0$. If $t_n = t^\star$ for some $n\in\NN$, then there is nothing to show. So suppose that $t_n < t^\star$ for all $n\in\NN$. Then by \eqref{eqn:lem:construction:pf:y:2}, the dominated convergence theorem (recall the integrability of $y$), and \eqref{eqn:lem:construction:pf:y:1},
\begin{align}
\label{eqn:lem:construction:pf:ystar:nonpositive}
y(t^\star)
=1-\int_{[t^\star, b)} y \dd\mu 
= 1 -\lim_{n\to\infty} \int_{(t_n, b)} y \dd\mu
= \lim_{n\to\infty} \big(1+\mu(\lbrace t_n \rbrace)\big)y(t_n)
\leq 0.
\end{align}
Similarly, we have $y(t^\star) \geq 0$ because
\begin{align}
\label{eqn:lem:construction:pf:ystar:nonnegative}
y(t^\star)\big(1+\mu(\lbrace t^\star \rbrace)\big)
&= 1- \int_{(t^\star,b)} y \dd \mu
= 1-\lim_{\varepsilon \downarrow\downarrow 0} \int_{[t^\star + \varepsilon, b)} y \dd \mu 
= \lim_{\varepsilon \downarrow\downarrow 0} y(t^\star + \varepsilon)
\geq 0,
\end{align}
where we use the definition of $t^\star$ in the last inequality. So $y(t^\star) = 0$ by \eqref{eqn:lem:construction:pf:ystar:nonpositive}--\eqref{eqn:lem:construction:pf:ystar:nonnegative}. As this also implies $\int_{[t^\star,b)} y \dd\mu = 1$ by \eqref{eqn:lem:construction:pf:y:2}, we infer from \eqref{eqn:lem:construction:pf:y:2} that
\begin{align*}
y(t)
&= 1 - \int_{[t,b)} y \dd\mu
= \int_{[t^\star,b)} y \dd\mu - \int_{[t,b)} y \dd\mu
= \int_{[t^\star,t)} y \dd\mu
= \int_{(t^\star,t)} y \dd\mu, \quad t\in (t^\star,b).
\end{align*}
But $\mu((t^\star, b)) < \infty$ by the assumption on $\mu$ (recall that $t^\star > a$). Hence, the Gronwall--Bellman lemma (Lemma~\ref{lem:gronwall}) gives $y \leq 0$ on $(t^\star,b)$. This contradicts the fact that $y(b-) = 1$. We conclude that $y$ is positive.

Third, we prove the remaining properties. The positivity of $y$ together with \eqref{eqn:lem:construction:pf:y:2} shows that $y$ is nondecreasing and has range $(0,1]$. Hence, the limit $y(a+)=\lim_{t \downarrow\downarrow a} y(t)$ exists in $[0,1]$. Now by \eqref{eqn:lem:construction:pf:y:2},
\begin{align}
\label{eqn:lem:construction:pf:convergence}
1-y(t)
&= \int_{[t,b)} y \dd\mu \geq y(a+) \mu([t,b)),\quad t\in(a,b).
\end{align}
Letting $t \downarrow \downarrow a$ in \eqref{eqn:lem:construction:pf:convergence} and using that $\mu((a,b)) = \infty$, we conclude that $y(a+) = 0$.
\end{proof}

Suppose that a Borel measure $\mu$ admits an $a \in [-\infty,\infty)$ such that $\mu((a,t)) = \infty$ for all $t > a$. If $\mu$ is semi-finite, the following lemma constructs a Borel set $E \subset (a,\infty)$ such that the measure $\mu_E$ defined by $\mu_E(A) = \mu(A\cap E)$, $A \in \cB(\RR)$, satisfies the assumptions of Lemma~\ref{lem:construction} (for any $b > a$).

\begin{lemma}
\label{lem:subset}
Let $\mu$ be a semi-finite Borel measure on $\RR$. Suppose that there is $a \in [-\infty,\infty)$ such that $\mu((a,t)) = \infty$ for all $t > a$. Then there is a Borel set $E \subset (a, \infty)$ such that the Borel measure $\mu_E$ defined by $\mu_E(A) = \mu(A\cap E)$, $A \in \cB(\RR)$, has the following property:
\begin{align*}
\text{For each }b\in(a,\infty)\text{ and }t\in[a,b),~\mu_E((t,b)) = \infty \text{ if and only if }t = a.
\end{align*}
\end{lemma}

\begin{proof}
We start by constructing a candidate for the set $E$. To this end, first choose a bi-infinite increasing sequence $(t_n)_{n\in\ZZ} \subset (a,\infty)$ such that $\lim_{n\to-\infty} t_n = a$ and $\lim_{n\to\infty} t_n = \infty$, set $G_n := (t_{n-1},t_n]$ for $n \in \ZZ$, and denote by $I\subset\ZZ$ the set of indices $n$ for which $\mu(G_n) = \infty$.
%Moreover, as $\mu$ is semi-finite, there is an increasing sequence $(A_m)_{m\in\NN} \subset \cB(\RR)$ such that $\bigcup_{m\in\NN} A_m = \RR$ and $\mu(A_m) < \infty$ for each $m \in \NN$.
By Proposition~\ref{prop:semi-finite}~(c), there is, for each $n \in I$, a Borel set $F_n \subset G_n$ such that $1 < \mu(F_n) < \infty$. For $n \in \ZZ \setminus I$, we set $F_n = G_n$. Note that with these definitions,
\begin{align}
\label{eqn:lem:subset:pf:finite}
\mu(F_n)
&<\infty, \quad n\in\ZZ.
\end{align}
Finally, we set $E = \bigcup_{n \in \ZZ} F_n$. (Note that $E$ is a disjoint union and that $E = (a,\infty)$ if and only if $I$ is empty.)

Now we verify that $E$ has the asserted properties. Fix $b \in (a,\infty)$. If $t \in (a,b)$, then
\begin{align*}
\mu_E((t,b))
&= \mu((t,b)\cap E)
= \sum_{n\in\ZZ} \mu\left((t,b)\cap F_n\right)
\end{align*}
is finite because of \eqref{eqn:lem:subset:pf:finite} and the fact that $(t,b) \cap F_n = \emptyset$ for all but finitely many $n \in \ZZ$. It thus remains to show that $\mu_E((a,b))=\infty$. We distinguish two cases. First, suppose that the set $I^- := \lbrace n \in I: n < 0\rbrace$ is finite. Then there is $N\in\ZZ$ such that $\mu(G_n) < \infty$ for all $n \leq N$, which in turn gives $E \cap (a,t_N] = (a,t_N]$. Making $N$ smaller if necessary, we may assume that $t_N < b$. Then
\begin{align*}
\mu_E((a,b))
&\geq \mu((a,t_N]\cap E)
= \mu((a,t_N])
= \infty,
\end{align*}
where we use the assumption on $\mu$ in the last equality. Second, suppose that the set $I^-$ is infinite. Then recalling that $\mu(F_n) > 1$ for $n \in I$, we obtain
\begin{align*}
\mu_E((a,b))
&= \mu((a,b)\cap E)
\geq \sum_{\substack{n \in I^-:\\t_n < b}} \mu\left(F_n\right)
= \infty.\qedhere
\end{align*}
\end{proof}

The implication ``(I$_a$) $\Rightarrow$ (M$_a$)'' of Theorem~\ref{thm:local} is proved by contraposition: Assuming that (M$_a$) fails, we want to construct a nonnegative Borel function $y$ which is $\mu$-integrable over $(a,b)$, positive with positive $\mu$-measure, and still satisfies the integral inequality \eqref{eqn:thm:local:inequality}. The next result shows that in these properties of $y$, we may replace $\mu$ by its semi-finite part $\mu_\sef$.

\begin{lemma}
\label{lem:mu sf to mu}
Let $\mu$ be a Borel measure on $\RR$ and fix $-\infty\leq a<b\leq \infty$. Then the following are equivalent:
\begin{enumerate}
\item There is a nonnegative $\mu$-integrable Borel function $y:(a,b)\to [0,\infty)$ which is positive with positive $\mu$-measure and satisfies 
\begin{align}
\label{eqn:lem:mu sf to mu:mu}
y(t)
&\leq \int_{(a,t)} y \dd\mu \quad\text{for } \mu\text{-a.e.~} t\in(a,b).
\end{align}
\item There is a nonnegative $\mu_\sef$-integrable Borel function $y_\sef:(a,b)\to[0,\infty)$ which is positive with positive $\mu_\sef$-measure and satisfies
\begin{align}
\label{eqn:lem:mu sf to mu:mu sf}
y_\sef(t)
&\leq \int_{(a,t)} y_\sef \dd\mu_\sef \quad\text{for } \mu_\sef\text{-a.e.~} t\in(a,b).
\end{align}
\end{enumerate}
\end{lemma}

\begin{proof}
``(a) $\Rightarrow$ (b)'': Let $y$ be as in (a) and set $y_\sef := y$. Applying Proposition~\ref{prop:semi-finite}~(a) and (b) to \eqref{eqn:lem:mu sf to mu:mu} shows that $y_\sef$ satisfies \eqref{eqn:lem:mu sf to mu:mu sf}. Moreover, $\lbrace y > 0 \rbrace$ is not a $\mu$-nullset by assumption. Hence, by Corollary~\ref{cor:semi-finite}, $\lbrace y > 0 \rbrace$ is also not a $\mu_\sef$-nullset. Therefore, $y_\sef = y$ is positive with positive $\mu_\sef$-measure. So (b) holds.

``(b) $\Rightarrow$ (a)'': Let $y_\sef$ be as in (b). As $y_\sef$ is $\mu_\sef$-integrable over $(a,b)$, there is by Proposition~\ref{prop:semi-finite}~(b$^\prime$) a $\mu$-integrable Borel function $\ol y:(a,b)\to \RR$ with $\ol y = y_\sef$ $\mu_\sef$-a.e.~on $(a,b)$.

We now use $\ol y$ to construct a nonnegative Borel function $y$ which coincides with $y_\sef$ $\mu_\sef$-a.e.~on~$(a,b)$, is $\mu$-integrable over $(a,b)$, and satisfies the integral inequality $y(t) \leq \int_{(a,t)} y \dd \mu$ \emph{for each} $t \in (a,b)$. To this end, we truncate $\ol y$ as follows. Define the function $y:(a,b)\to[0,\infty)$ by
\begin{align*}
y(t)
&= \min\Big(\max(\ol y(t),0), \int_{(a,t)} y_\sef \dd\mu_\sef \Big).
\end{align*}
As $\ol y = y_\sef$ $\mu_\sef$-a.e.~on $(a,b)$ and $0 \leq y_\sef \leq \int_{(a,\cdot)}y_\sef\dd\mu_\sef$ $\mu_\sef$-a.e.~on $(a,b)$~by (b), we find that $y = y_\sef$ $\mu_\sef$-a.e.~on $(a,b)$. Moreover, $0 \leq y \leq \max(\ol y,0)$ by construction, so that $y$ is $\mu$-integrable over $(a,b)$. In addition, by construction of $y$, the fact that $y = y_\sef$ $\mu_\sef$-a.e.~on $(a,b)$, and Proposition~\ref{prop:semi-finite}~(b), the integral inequality for $y$ obtains:
\begin{align*}
y(t)
&\leq \int_{(a,t)}y_\sef \dd\mu_\sef
= \int_{(a,t)} y \dd\mu_\sef
= \int_{(a,t)} y \dd\mu,\quad t\in(a,b).
\end{align*}

Finally, as the set $\lbrace y > 0 \rbrace$ has positive $\mu_\sef$-measure, it also has positive $\mu$-measure by absolute continuity of $\mu_\sef$ with respect to $\mu$ (Proposition~\ref{prop:semi-finite}~(a)).
\end{proof}

The final result of this section is used only in the counter-example in Remark~\ref{rem:counter-examples}~(c) and not in the proofs of the main results in the following Section~\ref{sec:proofs}. It provides a constant-sign property for solutions to the homogeneous linear integral \emph{equation} \eqref{eqn:lem:integral equation} for general Borel measures on $\RR$.

\begin{lemma}
\label{lem:integral equation}
Let $\mu$ be a Borel measure on $\RR$ and fix $-\infty\leq a<b\leq \infty$. Suppose that $y:(a,b)\to\RR$ is $\mu$-integrable over $(a,t)$ for each $t \in (a,b)$ and solves the integral equation
\begin{align}
\label{eqn:lem:integral equation}
y(t)
&= \int_{(a,t)} y \dd\mu, \quad t\in(a,b).
\end{align}
Then $y$ has constant sign ($y \geq 0$ on $(a,b)$ or $y \leq 0$ on $(a,b)$) and is monotone (nondecreasing or nonincreasing).
\end{lemma}

Note that the integral equation \eqref{eqn:lem:integral equation} may have nontrivial solutions. For example, if $a=0$, $b=\infty$, and $\mu(\diff t) = \frac{1}{t}\1_{(0,\infty)}(t) \dd t$, then the identity function $y(t) = t$, $t \in (0,\infty)$, solves \eqref{eqn:lem:integral equation}.

\begin{proof}
To prove that $y$ has constant sign, it suffices to show that $y$ has the following property: if $y(t) \neq 0$ for some $t \in (a,b)$, then for every $t' > t$ in $(a,b)$, $y(t')$ is nonzero and has the same sign as $y(t)$. Fix $t \in (a,b)$ such that $y(t) \neq 0$. As $-y$ also solves \eqref{eqn:lem:integral equation}, we may assume without loss of generality that $y(t) > 0$. By \eqref{eqn:lem:integral equation},
\begin{align}
\label{eqn:lem:integral equation:pf:calculation}
\begin{split}
y(u)
&= \int_{(a,u)} y \dd\mu
= \int_{(a,t)} y \dd\mu + y(t)\mu(\lbrace t \rbrace) +  \int_{(t,u)} y \dd\mu\\
&= y(t)(1+\mu(\lbrace t \rbrace)) + \int_{(t,u)} y \dd\mu, \quad u \in (t,b).
\end{split}
\end{align}
Seeking a contradiction, suppose that there is $t' \in(t,b)$ such that $y(t') \leq 0$. Then
\begin{align*}
t^\star
&:= \inf\lbrace t'\in (t,b): y(t') \leq 0 \rbrace \in [t,b).
\end{align*}
We now distinguish two cases. First, suppose that $t^\star$ attains the infimum, i.e., $y(t^\star) \leq 0$. Then $t^\star > t$ (because $y(t) > 0$) and by definition of $t^\star$, $y > 0$ on $[t,t^\star)$. Thus, using \eqref{eqn:lem:integral equation:pf:calculation} for $u = t^\star$ yields
\begin{align*}
0
\geq y(t^\star)
= y(t)(1+\mu(\lbrace t \rbrace)) + \int_{(t,t^\star)} y \dd\mu
>0,
\end{align*}
which is absurd. Second, suppose that $t^\star$ does not attain the infimum, i.e., $y(t^\star) > 0$. Then there is a sequence $(t_n)_{n\in\NN}\subset(t^\star,b)$ decreasing to $t^\star$ such that $y(t_n) \leq 0$. Then by \eqref{eqn:lem:integral equation:pf:calculation} for $u = t_n$,
\begin{align*}
0
&\geq y(t_n)
= y(t)(1+\mu(\lbrace t \rbrace) + \int_{(t,t_n)} y \dd\mu, \quad n\in\NN.
\end{align*}
Letting $n\to\infty$, using dominated convergence, and noting that $\bigcap_{n \in \NN} (t,t_n) = (t,t^\star]$ because $t_n > t^\star$ for all $n$ yields
\begin{align*}
0
&\geq y(t)(1+\mu(\lbrace t\rbrace)) + \int_{(t,t^\star]} y \dd\mu > 0,
\end{align*}
which is again a contradiction. We conclude that $y(t') > 0$ for all $t' > t$ in $(a,b)$.

The monotonicity of $y$ now follows immediately from the constant-sign property and \eqref{eqn:lem:integral equation}.
\end{proof}

%=================================================================================
\section{Proofs of the main results}
%=================================================================================
\label{sec:proofs}

We are now in a position to prove our main results. We begin with the ``local'' version.

\begin{proof}[Proof of Theorem~\ref{thm:local}]
``(M$_a$) $\Rightarrow$ (I$_a$)'': Suppose that $\mu_\sef((a,t)) < \infty$ for some $t > a$, set $b = t$, and let $y$ be as in (I$_a$). We proceed to show that $y \leq 0$ $\mu$-a.e.~on $(a,b)$. By \eqref{eqn:thm:local:inequality} and Proposition~\ref{prop:semi-finite}~(b),
\begin{align}
\label{eqn:thm:local:pf:inequality}
y(t)
&\leq \int_{(a,t)} y \dd\mu_\sef \quad \text{for } \mu\text{-a.e.~} t \in (a,b).
\end{align}
Because $\mu_\sef$ is absolutely continuous with respect to $\mu$ (Proposition~\ref{prop:semi-finite}~(a)), \eqref{eqn:thm:local:pf:inequality} also holds $\mu_\sef$-a.e.~on $(a,b)$. Thus, we may apply the Gronwall--Bellman lemma (Lemma~\ref{lem:gronwall}) to $y$ and $\mu_\sef$, which yields $y \leq 0$ $\mu_\sef$-a.e.~on $(a,b)$. Combining this with \eqref{eqn:thm:local:pf:inequality} then gives $y \leq 0$ $\mu$-a.e.~on $(a,b)$ (alternatively, one can also invoke Corollary~\ref{cor:semi-finite} for this last step).

``(I$_a$) $\Rightarrow$ (M$_a$)'': We prove the contrapositive. Suppose that $\mu_\sef((a,t)) = \infty$ for every $t>a$, and fix any $b\in(a,\infty)$. By Lemma~\ref{lem:mu sf to mu}, it suffices to construct a nonnegative $\mu_\sef$-integrable Borel function $y_\sef$ on $(a,b)$ which is positive with positive $\mu_\sef$-measure and satisfies
\begin{align}
\label{eqn:thm:local:pf:integral inequality}
y_\sef(t)
&\leq \int_{(a,t)} y_\sef \dd\mu_\sef \quad\text{for }\mu_\sef \text{-a.e.~}t \in (a,b).
\end{align}

By Lemma~\ref{lem:subset} (applied to $\mu_\sef$), there is a Borel set $E \subset (a,\infty)$ such that the Borel measure $\mu_{\sef,E}$ (defined as in Lemma~\ref{lem:subset}) fulfils the assumption of Lemma~\ref{lem:construction} (with $\mu$ replaced by $\mu_{\sef,E}$). By Lemma~\ref{lem:construction}, there is then a Borel function $\ol y_\sef :(a,b)\to(0,1]$ such that
\begin{align}
\label{eqn:thm:local:pf:y}
\ol y_\sef (t)
&= \int_{(a,t)} \ol y_\sef \dd\mu_{\sef,E}, \quad t\in(a,b),
\quad\text{and}\quad \lim_{t \uparrow \uparrow b} \ol y_\sef(t) = 1.
\end{align}
Now define $y_\sef:\RR \to [0,1]$ by $y_\sef = \ol y_\sef \1_E$ on $(a,b)$ and $0$ elsewhere. We first show that $y_\sef$ is $\mu_\sef$-integrable on $\RR$. By \eqref{eqn:thm:local:pf:y} and the construction of $\mu_{\sef,E}$,
\begin{align*}
\int_{\RR} y_\sef \dd\mu_\sef
&= \int_{(a,b)} \ol y_\sef \1_E \dd\mu_\sef
= \int_{(a,b)} \ol y_\sef \dd\mu_{\sef,E}
= 1.
\end{align*}
Moreover, $y_\sef$ is positive with positive $\mu_\sef$-measure because $\ol y_\sef$ is positive on $(a,b)$ and $\mu_\sef$ has positive (even infinite) mass on $(a,b) \cap E$ ($\mu_{\sef,E}$ has infinite mass on $(a,b)$ by construction).

It remains to show \eqref{eqn:thm:local:pf:integral inequality}. This follows from \eqref{eqn:thm:local:pf:y} and the construction of $\mu_{\sef,E}$ and $y_\sef$:
\begin{align*}
y_\sef(t)
&\leq \ol y_\sef(t)
= \int_{(a,t)} \ol y_\sef  \dd\mu_{\sef,E}
= \int_{(a,t)} \ol y_\sef \1_E \dd\mu_\sef
= \int_{(a,t)} y_\sef \dd\mu_\sef,\quad t \in (a,b).\qedhere
\end{align*}
\end{proof}

We finally prove the ``global'' version of our main result.

\begin{proof}[Proof of Theorem~\ref{thm:global}]
We first show that $\mu_\sef$ is $\sigma$-finite if (M) holds. Since $\mu_\sef$ is semi-finite, every singleton has finite $\mu_\sef$-measure. This together with (M) implies that for every $a \in \RR$, there is $\varepsilon>0$ such that $\mu_\sef([a,a+\varepsilon)) < \infty$. Thus, $\mu_\sef$ is $\sigma$-finite by Lemma~\ref{lem:sigma finite}.

``(M) $\Rightarrow$ (I)'': Suppose that for each $a \in[-\infty,\infty)$, there is $t>a$ such that $\mu_\sef((a,t)) < \infty$, and let $y$ be as in (I). We need to show that $y \leq 0$ $\mu$-a.e. It follows from Corollary~\ref{cor:semi-finite} that it is enough to show that $y \leq 0$ $\mu_\sef$-a.e. Consider
\begin{align*}
t^\star
&= \sup \lbrace t \in [-\infty,\infty] : y\leq0~\mu_\sef\text{-a.e.~on }(-\infty,t) \rbrace.
\end{align*}

We first show that $t^\star > -\infty$. By (M), there is $b \in \RR$ such that $\mu_\sef((-\infty,b)) < \infty$. Hence, by Lemma~\ref{lem:gronwall} (for $a = -\infty$) and \eqref{eqn:thm:global:inequality}, we have $y \leq 0$ $\mu_\sef$-a.e.~on $(-\infty,b)$, so that $t^\star \geq b$.

Next, we show that $t^\star = \infty$. Seeking a contradiction, suppose that $t^\star < \infty$ and let $(t_n)_{n\in\NN} \subset (-\infty,t^\star]$ be a sequence increasing to $t^\star$ such that $y \leq 0$ $\mu_\sef$-a.e.~on $(-\infty,t_n)$ for all $n \in \NN$ (this exists because $t^\star > -\infty$). Then also $y \leq 0$ $\mu_\sef$-a.e.~on the countable union $\bigcup_{n\in\NN} (-\infty, t_n) = (-\infty,t^\star)$ (i.e., $t^\star$ attains the supremum). Now by (M), there is $b \in \RR$ such that $\mu_\sef((t^\star,b)) < \infty$. Moreover, by \eqref{eqn:thm:global:inequality} and the fact that $y \leq 0$ $\mu_\sef$-a.e.~on $(-\infty,t^\star)$ we first find that $y(t^\star) \leq 0$ if $\mu_\sef$ has a point mass at $t^\star$, and then, for $t \in (t^\star,b)$,
\begin{align*}
y(t)
&\leq \int_{(-\infty,t)} y \dd\mu_\sef
\leq\int_{[t^\star,t)} y \dd\mu_\sef
\leq \int_{(t^\star,t)} y \dd\mu_\sef.
\end{align*}
We may thus apply the Gronwall--Bellman lemma (Lemma~\ref{lem:gronwall}) and conclude that $y \leq 0$ $\mu_\sef$-a.e.~on $(t^\star, b)$. Because we also have $y \leq 0$ $\mu_\sef$-a.e.~on $(-\infty, t^\star)$ and $y(t^\star) \leq 0$ if $\mu_\sef$ has a point mass at $t^\star$, we obtain that $y \leq 0$ $\mu_\sef$-a.e.~on $(-\infty,b)$. As $b > t^\star$, this contradicts the definition of $t^\star$. Thus $t^\star = \infty$, which means $y \leq 0$ $\mu_\sef$-a.e.

``(I) $\Rightarrow$ (M)'': We prove the contrapositive. Suppose that (M) fails. Then there is $a \in [-\infty,\infty)$ such that $\mu_\sef((a,t)) = \infty$ for all $t > a$. Choose any $b > a$. Then by Theorem~\ref{thm:local}, there is a real-valued Borel function $y$ on $\RR$ which is $\mu$-integrable over $(a,b)$ and satisfies
\begin{align}
\label{eqn:thm:global:pf:local integral inequality}
y(t)
&\leq \int_{(a,t)} y \dd \mu \quad \text{for } \mu\text{-a.e.~} t \in (a,b)
\end{align}
and $\mu(\lbrace t \in (a,b): y(t) > 0\rbrace) > 0$. Replacing $y$ by its positive part $\max(y,0)$ if necessary, we may assume that $y$ is nonnegative (note in particular that \eqref{eqn:thm:global:pf:local integral inequality} still holds for the positive part; cf.~Remark~\ref{rem:nonnegative}). Then the Borel function $y':= y\1_{(a,b)}$ is nonnegative on $\RR$ and positive with positive $\mu$-measure and satisfies $\int_{(-\infty,t)} y'\dd\mu <\infty$ for all $t\in\RR$. Thus, it remains to show that $y'(t) \leq \int_{(-\infty,t)} y' \dd\mu$ for $\mu$-a.e.~$t\in\RR$. As $y'$ is nonnegative and vanishes off $(a,b)$, this inequality is trivially satisfied for $t \in \RR\setminus(a,b)$. Moreover, using that $y'=y$ on $(a,b)$ by construction and \eqref{eqn:thm:global:pf:local integral inequality}, we obtain 
\begin{align*}
y'(t)
&= y(t)
\leq \int_{(a,t)} y \dd\mu
= \int_{(a,t)} y' \dd\mu
= \int_{(-\infty,t)} y' \dd\mu\quad\text{for }\mu\text{-a.e.~}t \in (a,b).
\end{align*}
We conclude that (I) fails.
\end{proof}

%=================================================================================
\small
\providecommand{\bysame}{\leavevmode\hbox to3em{\hrulefill}\thinspace}
\providecommand{\MR}{\relax\ifhmode\unskip\space\fi MR }
% \MRhref is called by the amsart/book/proc definition of \MR.
\providecommand{\MRhref}[2]{%
  \href{http://www.ams.org/mathscinet-getitem?mr=#1}{#2}
}
\providecommand{\href}[2]{#2}

%=================================================================================
\end{document}